\documentclass[12pt,titlepage,a4paper,draft]{article}

\usepackage{amssymb,latexsym,amsfonts,amsmath,amsthm, url,
verbatim}

\usepackage{color}

{\makeatletter \@addtoreset{equation}{section}}

\definecolor{Royalblue}{cmyk}{1,0.30,0.2,0.2}

\def\bmat{\left[ \begin{array}}
\def\emat{\end{array} \right]}

\newtheorem{theorem}{Theorem}[section]
\newtheorem{lemma}[theorem]{Lemma}

\newtheorem{corol}[theorem]{Corollary}

\newcommand{\C}{{\mathbb{C}}}
\newcommand{\R}{{\mathbb{R}}}

\newcommand{\Ker}[1][\,]{\mathop{\mathrm{Ker \!}#1}\nolimits}

\newcommand{\IIm}[1][\,]{\mathop{\mathrm{Im}#1}\nolimits}

\newcommand{\bcmm}{\begin{comment}}
\newcommand{\ecmm}{\end{comment}}

\newcommand{\bbm}{\begin{bmatrix}}
\newcommand{\ebm}{\end{bmatrix}}

\newcommand{\beq}{\begin{equation}}
\newcommand{\eeq}{\end{equation}}
\newcommand{\bea}{\begin{eqnarray}}
\newcommand{\eea}{\end{eqnarray}}

\newcommand{\bsea}{\begin{subeqnarray}}
\newcommand{\esea}{\end{subeqnarray}}

\def\bmat{\left[ \begin{array}}
\def\emat{\end{array} \right]}

\def\bmat{\left[ \begin{array}}
\def\emat{\end{array} \right]}
\definecolor{darkgreen}{rgb}{0,0.7,0}
 \definecolor{exp}{rgb}{0.5,0.2,0.4}
 \definecolor{expold}{rgb}{0.85,0.7,0}
  \definecolor{msk}{rgb}{0.65,0.3,0.3}

\definecolor{msk2}{rgb}{0.75,0.3,0.1}
\definecolor{r}{rgb}{0.75,0.1,0.4}

\definecolor{exp2}{rgb}{0.85,0.6,0.1}

\definecolor{exp4}{rgb}{0.85,0.3,0.5}

\definecolor{vvs}{rgb}{0.5,0.4,0.7}

\begin{document}

\title{Pairs of $k$-step reachability and $m$-step observability matrices}

\author{ Augusto Ferrante\\
Dipartimento di Ingegneria dell'Informazione\\
Universit\`{a} di Padova\\
via Gradenigo, 6B\\
35131 Padova, Italy \\
\texttt{augusto@dei.unipd.it}
\and
%%%
Harald~K. Wimmer
                       \\
Mathematisches Institut\\
Universit\"at W\"urzburg\\
97074 W\"urzburg, Germany\\
\texttt{wimmer@mathematik.uni-wuerzburg.de}
}
\date{\today}

\maketitle

\begin{abstract}
 
Let  $V$ and $W$ be matrices of size  
$ n \times pk$ and $q m \times n   $, respectively.
A necessary and sufficient condition is
given for the existence of a triple $(A,B,C)$
such that  
$V$ a $k$-step reachability matrix  of $(A,B)$
and $W$ an $m$-step observability matrix   of  $(A,C)$. 

\vspace{2cm}

\tolerance 200

%\vskip 2cm
%\noindent
%{\bf Corresponding Author:}\\
%...
%

%\flushleft{
%{\textsf e-mail:}~~\texttt{\small augusto@dei.unipd.it} }
%%
%\flushleft{
%{\textsf e-mail:}
%~~\texttt{\small wimmer@mathematik.uni-wuerzburg.de} }
%%}

\end{abstract}

%%\setpagewiselinenumbers
%%\\pagewiselinenumbers
%\modulolinenumbers[5]
%\linenumbers

%%\def\linenumberfont{\normalfont\small\sffamily}

\section{Introduction}  
Let $A \in K^{n \times n}$, $B \in K^{n \times p} $, $C \in K^{q \times n } $ 
be matrices over a field $K$ { (this is done for the sake of generality: in practical situations one usually has $K=\R$ of $K=\C$)}. 
The matrix 
\[  
R_k (A, B) = \bbm
 B & AB & \dots & A^{k-1} B \ebm  \: \in  K^{n\times  p k } 
%% K^{n\times (  p \cdot k) }  
\]
is the {\em{$k$-step reachability matrix}} associated to $ (A,B)  $,
 and  
%%  (see e.g.\ \cite[Chapter 24]{DDV} or \cite[p.\ 705]{dNFT})   the matrix 
\[
%S_k(A, C) = 
O_m(A, C)  = 
 \bbm  C \\ CA \\ \vdots \\ C A ^{m-1} \ebm \: \in   
K^{ q m  \times n   } 
%%K^{ (q \cdot m)  \times n   } 
\]
is the {\em{$m$-step  observability  matrix}} associated to $ (A,C) $. 
%\marginpar{stl}
\\
%\\
%========
%\\
%These matrices are an important  tool in linear systems theory
%%%Applications of these matrices in   linear systems theory can be found  in
%(see e.g.\ \cite{DDV},  \cite{KKE}, \cite{So}).
%%%We refer to  \cite{DDV},  \cite{KKE}, \cite{So} 
%%(see e.g.\ \cite{DDV},  \cite{KKE}, \cite{So}). 
%%%% positive systems Benvenuti ... De Santis 
%%%
%In particular, an important application of the problem here considered may be found in identification:
% In the last decade a family of subspace-based methods for the identification 
%of MIMO linear time-invariant systems have been proposed, see \cite{QIN} for an overview and a
% rather extensive list of references. 
%Essentially, these methods consist of two steps:\\
%1. First the observability and reachability matrices of the systems are estimated.\\
%2. The matrices $(A,B,C)$ of a state space realization of the system are 
%than recovered from  the observability and reachability matrices.
%%\\
%
%Therefore, given pair of matrices $ V \in  K^{n\times pk}   $ and $ W \in K^{ q m \times n }  $, it very natural to ask whether or not  these matrices are a $k$-step reachability matrix and   an 
% $m$-step observability matrix, respectively, for some triple $(A,B,C)$ and, if the answer is positive, to establish a parametrization of all the possible triples  
%$(A,B,C)$ compatible with the given matrices.
%This is the problem addressed in this short note.
%\\
%=========
%
%\medskip
%\noindent
%%%%%%%%%%%%% 
%\textcolor{msk}{
These matrices are a basic  tool in linear systems theory 
(see e.g.\ \cite{DDV},  \cite{KKE}, \cite{So}).   The fact that they play an important role in identification
provides the motivation for our 
study.  In the last decade a family of subspace-based methods for the identification 
of MIMO linear time-invariant systems have been proposed (see \cite{QIN} for an overview).
Essentially, these methods consist of two steps: 
First the observability and reachability matrices of the systems are estimated. 
Then the matrices $(A,B,C)$ of a state space realization of the system are 
recovered from  the observability and reachability matrices.
%}
%
%\textcolor{msk}{
Therefore, given a pair of matrices $ (V , W ) \in  K^{n\times pk}    \times K^{ q m \times n } $
%%\in  K^{n\times pk}   $ and $ W \in K^{ q m \times n }  $,
 it is natural to ask whether   $ V $ is a $k$-step reachability matrix and  
$ W $ is  an  $m$-step observability matrix for some triple $(A,B,C)$, 
and if the answer is positive, to establish a parametrization of all possible triples  
$(A,B,C)$ compatible with the given pair.
This is the problem addressed in this short note.%}

%\begin{comment} 
%\medskip
%\noindent
%========
%
%Therefore, given pair of matrices $ V \in  K^{n\times pk}   $ and $ W \in K^{ q m \times n }  $, it very natural to ask whether or not  these matrices are a $k$-step reachability matrix and   an 
% $m$-step observability matrix, respectively, for some triple $(A,B,C)$ and, if the answer is positive, to establish a parametrization of all the possible triples  
%$(A,B,C)$ compatible with the given matrices.
%This is the problem addressed in this short note.
%\\
%========
%
%\begin{comment} 
%
%\textcolor{msk}{
%Therefore, given a pair of matrices $ (V , W ) \in  K^{n\times pk}    \times K^{ q m \times n } $
%%%\in  K^{n\times pk}   $ and $ W \in K^{ q m \times n }  $,
% it is natural to ask whether   $ V $ is a $k$-step reachability matrix and  
%$ W $ is  an  $m$-step observability matrix for some triple $(A,B,C)$, 
%and if the answer is positive, to establish a parametrization of all possible triples  
%$(A,B,C)$ compatible with the given pair.
%This is the problem addressed in this short note.
%}
%\end{comment} 

\section{The result}  \label{sct.rslt} 
Our approach involves  %%generalized 
$\{1 \}$-inverses
and a common solution of a pair of matrix equations. 
%%We use the notation of \cite{BG}.  
If  $ F \in  K ^{s \times t} $ then,  according to  the notation of \cite{BG},  we set 
\,$ F^{\{  1 \} } = 
\{ Y \in  K ^{t \times s} ; \;   F Y F = F \} $.
A matrix $ Y \in F^{\{  1 \} } $ is said to be a  
{\em{$\{1\}$-inverse}} of $F$ and denoted by  
$ F^{( 1 ) } $.  
The following lemma is well known  (see \cite [p.\ 54/55]{BG}, \cite[p.\,49] {RM}).

\begin{lemma}  
%{\rm{\cite[p.\ 54/55]{BG}}}    %% Ex. 4, p. 54,  Ex. 5, p.55  
\label{la.bgv}
Let $F \in K ^{s \times t}$, 
$ H \in K ^{p \times q}$, 
$ C \in 
K ^{s \times p}  , D \in  K ^{t \times q}
 $. 
%% itemize 
{\rm{(i)}} The pair of  equations
\beq \label{eq.sm}  
 FX = C   \quad  and \quad X H = D 
\eeq 
have a common solution  $ X \in  K ^{t \times p}  $
if   and only if each equation 
separately has a solution and
\beq \label{eq.vsr}
CH =   F D . 
\eeq
{\rm{(ii)}}
Suppose  the equations \eqref{eq.sm} have a common solution.
Let $ F ^{(1)} \in  F ^{ \{1 \} } $ and 
$ H^{(1) } \in  H ^{\{1\}}$. 
Then  
\[
X_0 = F ^{(1)} C +  (  I  -  F ^{(1) }F ) D  H^{(1) } 
\]
 is a common solution of  \eqref{eq.sm}.
Moreover, 
$X_0 =  D  H^{(1) } +  F ^{(1)} C  ( I - H H^{(1) } )  $.
%%$X_0 \in K^{t \times p} $. Then
\\
{\rm{(iii)}}
If $Y_0 $ is a common solution of  \eqref{eq.sm}
 then the  general solution is 
\beq \label{eq.gnsl} 
X = Y_0 + (  I  - F  F ^{(1)} )  Z ( I - H^{(1) } H )  
\eeq 
for arbitrary $Z \in K^{s \times q } $.
\end{lemma} 

Condition \eqref{eq.vsr} can be traced back to a 1910 paper of  
 Cecioni~\cite{Ce}. 
In the following theorem it appears in the form \eqref{eq.nsc-co3}  
  and provides the required interlocking  
between reachability and observability matrices.

%%Our main result is the following. 
%%The block structure of $R $  and $O$

\begin{theorem} \label{prop.2} 
Let $ V \in   K^{n\times pk}  $ and $ W  \in    K^{n\times  qm}  $
be given, and
let 
\[
 V=\bbm V_0 &  V_1 & \dots &   V_{k-1} \ebm  
\quad and \quad 
W=\bbm W_0^{T} &  W_1^{T} & \dots &   W_{m -1}^{T} \ebm  ^{T} 
\]
be partitioned into blocks $V_i \in   K^{n \times p}$   %% $ i = 0, \dots , k-1$, 
and 
$W_i\in K^{q\times n}$, respectively. 
Set 
\[
V_{L} = \bbm V_0 &  V_1 & \dots &   V_{k- 2} \ebm
\quad 
and  \quad V^U =  \bbm V_1&  V_2 & \dots &   V_{k-1} \ebm , 
\]
and 
\[
W_{L} = \bbm  W_0  \\ W_1 \\ \vdots \\   W_{m-  2} \ebm
\quad 
and  \quad   W^U =  \bbm W_1\\  W_2 \\\vdots \\  W_{m-1}   \ebm .
\]
{\rm{(i)}}
Then there exists a triple 
$ (A,B,C) \in K^{n \times n } \times K^{n \times p } \times K^{q \times n} $
such that 
\beq \label{eq.rsbd} 
 V = R_{k} (A, B) \quad and \quad W = O_{m} (A, C) 
\eeq 
if and only if
\bea    \label{eq.nsc-co1}
\Ker   V _{L}  \subseteq  \Ker  V^U, \\
\label{eq.nsc-co2}
\IIm   W _{L} \supseteq  \IIm  W^{U}, \eea 
and 
\beq 
\label{eq.nsc-co3} 
W_{L}V^{U} =   W^U V _{L}  . 
\eeq 
{\rm{(ii)}}
Suppose the conditions 
 \eqref{eq.nsc-co1} - \eqref{eq.nsc-co3} are satisfied. 
Then we have 
 \eqref{eq.rsbd}  if and only if 
\[
A =  W_{L}^{(1)} W^{U} + 
(I - W_{L}^{(1)} W_{L} )   V^{U}  V_{L}  ^{(1) } 
+   (I - W_{L}^{(1)} W_{L}   ) Z ( I - V_L V_L ^{(1) }  ) , 
\, Z \in K ^{ n\times n} ,  
\]
and  $ (B, C ) = (V_0, W_0)$. 
\end{theorem}
\begin{proof}  
(i) 
The inclusion \eqref{eq.nsc-co1} holds if and only if  
\,$A _r   V_{L} =  V^{U} $\,
 for some 
$A_r  \in K^{n \times n} $.  
Hence  \eqref{eq.nsc-co1} is equivalent to
\,$
V_i = A _r V_{i-1} $, $i = 1, \dots , k $.
Thus we have  \eqref{eq.nsc-co1}  if and only if 
\[
V = \bbm V_0 &  A_r V_0 & \dots &   A_r ^{k-1} V_0  \ebm = R_k(A_r, V_0)
\]
for some $A _r\in K^{n \times n} $.
Similarly,  \eqref{eq.nsc-co2}
is valid   if and only if  
\,$   W_{L}  A_o=  W^{U} $\,
 for some 
$A_o  \in K^{n \times n} $.  
Hence \eqref{eq.nsc-co2} holds if and only if 
\,$
W = O_{m} (A_o, W_0) $\, 
for some $A _o\in K^{n \times n} $.
Now suppose 
$ V = R_k(A_r, V_0) $  and $ W = O_{m} (A_o, W_0) $. 
Then there exists a matrix $A$ such that $ A = A_r = A_o$
if and only if   
 each of the  two  matrix equations 
$X   V_{L} =  V^{U} $ and $   W_{L}  X=  W^{U} $
is consistent and 
%%there  exists a matrix $A$ such that $ A = A_r = A_o$ 
there exists a common solution. 
We apply 
  Lemma~\ref{la.bgv}(i)
and conclude that   the identity~\eqref{eq.nsc-co3}  
is  necessary and sufficient for the  existence of a common solution~$X$.
Part (ii) follows from  Lemma~\ref{la.bgv}(ii)
and (iii). 
\end{proof}

If we take $W= 0 $ in  Theorem~\ref{prop.2} 
then we obtain the following.
%% corollary.
%%
%% A special case is contained in \cite{FW}

\begin{corol} \label{cor.bl} 
Let $ V \in   K^{n\times (p\cdot k)}  $ be given.
Then there exists a pair $ (A, B) \in K^{n \times n } \times K^{n \times p }$
such that $ V =  R_{k} (A, B) $
if and only if 
\,$\Ker   V _{L}  \subseteq  \Ker  V^U$. 
In that case we have  $ V =  R_{k} (A, B) $
if and only if 
\[
A =   V^{U}  V_{L}^{(1)}+  Z( I - V_L V_L ^{(1) } ), 
\, \, Z \in K ^{ n\times n} ,  
\]
and  $ B = V_0$. 
\end{corol} 

For the  special  case $p=1$  and  $k = n$,  
and more detailed results involving  companion matrices,
 we refer to
\cite{FW}.

\begin{comment}

In \cite{FW} we obtained a special case. 
We considered 
matrices $V $ of size  $n \times k$ and
 $k$-step controllability matrices 
$R_k(A, b ) = (b, Ab , \dots , A^{k-1} b)$ with $b \in K^n$.

\end{comment}


\begin{thebibliography}{99}


\bibitem{BG} 
A.\ Ben-Israel and Th.\ N.\ E.\ Greville, Generalized Inverses,
Theory and Applications, 2nd edition, Springer, New York, 2003. 
%% p.52  AXB = D 

\bibitem{Ce}
 F. ~Cecioni, 
Sopra operazioni algebriche, Ann. Scuola Nom. Sup. Pisa Sci. Fis.
Mat. 11 (1910),   17--20.

%%%\bibitem{2B4} 
%%%A. J. Brzezinski, S.   Kukreja,   Jun Ni,  and  D. S.Bernstein,  
%%Pseudo transfer functions in sensor-only identification,
%%submitted for publication. 

\bibitem{DDV}
M.\ Dahleh, M.\ A.\ Dahleh, and G. Verghese, Lectures on Dynamic Systems and Control,
MIT Lectures, 2004. 
Available online: 
%%  
\\
\url{web.mit.edu/6.241/www/chapter_22.pdf}
\\
\url{web.mit.edu/6.241/www/chapter_26.pdf}  


\begin{comment} 
\bibitem{dNFT} 
G.\  De Nicolao and G.\ Ferrari Trecate, 
On the zeros of discrete-time linear periodic systems, 
J.\ Circuits Systems Signal Process.  16 (1997), 703--718. 
%% p.705 k-step obs 

\end{comment} 

\bibitem{FW} 
A.\ Ferrante and H.\ K.\ Wimmer, 
Reachability matrices and cyclic matrices,
Electron. J. Linear Algebra 20 (2010), 95--102.
%% \\\url{url}  


\bibitem{KKE} 
E.\ W.\  Kamen,   P.\ P.\ Khargonekar, and K.\ R.\ Poolla, 
SIAM J.  Control Optim. 23,  A transfer function approach to linear time-varying
discrete-time systems, 550--565,  1985.
%% p.552 

\begin{comment} 
D.~M.~Luchtenburg and C.~W.~Rowley,
Model reduction using snapshot-based realizations, 
preprint 2012, 
submitted to SICON.
\end{comment} 


\begin{comment} 
\bibitem{Mi} 
S.K.  Mitra,  The matrix equations $AX=C$, $XB=D$, 
Linear Algebra Appl., vol. 59, 1984, p.171--181. 
\end{comment} 

\bibitem{QIN}
S.J. Qin, An overview of subspace identification, Computers and Chemical Engineering, vol. 30, 2006, pp. 1502--1513.
 


\bibitem{RM}
 C.  R.  Rao and S. K. Mitra, Generalized Inverse of Matrices and Its Applications,
Wiley, New York, 1971.


%% \bibitem{WdK} Willigenburg, de Koning 
%% La. 2.3, p.109 

\bibitem{So}
E.\ D.\  Sontag,
Mathematical Control Theory: Deterministic Finite Dimensional Systems,
%% p.296
2nd edition., Springer, New York, 1998. 
\end{thebibliography}
\end{document}